\documentclass[12pt,fleqn,leqno]{amsart}
\usepackage{amsfonts,amssymb}
\usepackage{mathrsfs}
\usepackage{enumitem}

\usepackage[svgnames]{xcolor}
\usepackage[colorlinks,linkcolor=blue,citecolor=Green]{hyperref}
\usepackage{titlesec}
\titleformat{\section}[block]
 {\bfseries}
 {\thesection.}
 {\fontdimen2\font}
 {}

\newcommand{\periodafter}[1]{#1.}

\titleformat{\subsection}[runin]
 {\bfseries}
 {\thesubsection.}
 {\fontdimen2\font}
 {\periodafter}

\setlist{noitemsep}
\setenumerate{labelindent=\parindent,label=\upshape{(\alph*)}}

\newtheorem{theorem}{Theorem}[section]
\newtheorem{corollary}[theorem]{Corollary}
\newtheorem{lemma}[theorem]{Lemma}
\newtheorem{proposition}[theorem]{Proposition}

\theoremstyle{definition}

\newtheorem{question}{Question}

\renewcommand{\emptyset}{\varnothing}
\numberwithin{equation}{section}

\DeclareMathOperator{\uhr}{\upharpoonright}

\DeclareMathAlphabet{\mathpzc}{OT1}{pzc}{m}{it}
\DeclareMathOperator{\sel}{\mathpzc{V\mkern-5mu_{cs}}}
\DeclareMathOperator{\cut}{\mathpzc{ct}}
\DeclareMathOperator{\noncut}{\mathpzc{nct}}

\begin{document}

\author{Valentin Gutev}
  \address{Department of Mathematics, Faculty of Science, University of
     Malta, Msida MSD 2080, Malta}

\email{\href{mailto:valentin.gutev@um.edu.mt}{valentin.gutev@um.edu.mt}}

\subjclass[2010]{54B20, 54C65, 54D05, 54D30, 54F05, 54F65}

\keywords{Vietoris topology, continuous selection, weak selection,
  weakly orderable space, weakly cyclically orderable space}

\title{Hyperspace Selections Avoiding Points}

\begin{abstract}
  In this paper, we deal with a hyperspace selection problem in the
  setting of connected spaces. We present two solutions of this
  problem illustrating the difference between selections for the
  nonempty closed sets, and those for the at most two-point sets. In
  the first case, we obtain a characterisation of compact orderable
  spaces. In the latter case --- that of selections for at most
  two-point sets, the same selection property is equivalent to the
  existence of a ternary relation on the space, known as a cyclic
  order, and gives a characterisation of the so called weakly
  cyclically orderable spaces.
\end{abstract}

\date{\today}
\maketitle

\section{Introduction}

All spaces in this paper are infinite Hausdorff topological
spaces. For a space $X$, let $\mathscr{F}(X)$ be the set of all
nonempty closed subsets of $X$. Usually, we endow $\mathscr{F}(X)$
with the \emph{Vietoris topology} $\tau_V$, and call it the
\emph{Vietoris hyperspace} of $X$. Recall that $\tau_V$ is generated
by all collections of the form
\[
\langle\mathscr{V}\rangle = \left\{S\in \mathscr{F}(X) : S\subset
  \bigcup \mathscr{V}\ \ \text{and}\ \ S\cap V\neq \emptyset,\
  \hbox{whenever}\ V\in \mathscr{V}\right\},
\]
where $\mathscr{V}$ runs over the finite families of open subsets of
$X$. In the sequel, any subset $\mathscr{D}\subset \mathscr{F}(X)$
will carry the relative Vietoris topology as a subspace of the
hyperspace $(\mathscr{F}(X),\tau_V)$.  A map $f:\mathscr{D}\to X$ is a
\emph{selection} for $\mathscr{D}$ if $f(S)\in S$ for every
$S\in\mathscr{D}$. A selection $f:\mathscr{D}\to X$ is
\emph{continuous} if it is continuous with respect to the relative
Vietoris topology on $\mathscr{D}$, and we use $\sel[\mathscr{D}]$ to
denote the set of all \emph{Vietoris continuous selections} for
$\mathscr{D}$.\medskip

A space $X$ is \emph{orderable} (or \emph{linearly ordered}) if it is
endowed with the open interval topology $\mathscr{T}_\leq$ generated
by some linear order $\leq$ on $X$, called \emph{compatible} for
$X$. Subspaces of orderable spaces are not necessarily orderable, they
are termed \emph{suborderable}. A space $X$ is \emph{weakly orderable}
if there exists a coarser orderable topology on $X$, i.e.\ if there
exists a linear order $\leq$ on $X$ (called \emph{compatible} for $X$)
such that $\mathscr{T}_\leq\subset\mathscr{T}$, where $\mathscr{T}$ is
the topology of $X$. The weakly orderable spaces were introduced by
Eilenberg \cite{eilenberg:41} under the name of ``ordered''
topological spaces, and are often called ``Eilenberg orderable''. They
were called ``weakly orderable'' in
\cite{zbMATH03379800,zbMATH03355968,dalen-wattel:73,mill-wattel:81},
and in \cite{mill-wattel:81} was also proposed to abbreviate them as
\emph{KOTS}. \medskip

It was shown in \cite[Proposition 5.1]{MR3712970} that if $X$ is a
connected space and $\mathscr{F}(Z)$ has a continuous selection for
every connected subset $Z\subset X$ with $|X\setminus Z|\leq 1$, then
$X$ is compact and orderable.  However, this selection property
includes the special case of $Z=X$, i.e.\ that $\mathscr{F}(X)$ itself
has a continuous selection. Here, we are going to show that this
result is valid without explicitly requiring that $\mathscr{F}(X)$ has
a continuous selection, see Theorem
\ref{theorem-Sel-Cyclic-v2:1}. Based on this, we will obtain the
following characterisation of connected compact orderable spaces.

\begin{theorem}
  \label{theorem-sa_points-v10:1}
  A connected space $X$ is compact and orderable if and only if
  $\sel[\mathscr{F}(X\setminus \{p\})]\neq \emptyset$, for each
  $p\in X$.
\end{theorem}

Regarding the proper place of Theorem \ref{theorem-sa_points-v10:1},
let us explicitly remark that if $X$ is a space such that
$\mathscr{F}(Z)$ has a continuous selection for every $Z\subset X$
with $|X\setminus Z| \leq 2$, then $X$ is totally disconnected
\cite[Corollary 5.3]{MR3712970}. On the other hand, the hypothesis in
Theorem \ref{theorem-sa_points-v10:1} that $X$ is connected is
essential to conclude that it is compact.  In fact, we will obtain a
natural generalisation of the aforementioned result of
\cite{MR3712970} showing that the hyperspace selection property in
Theorem \ref{theorem-sa_points-v10:1} is possessed by a natural class
of non-compact totally disconnected spaces, see Theorem
\ref{theorem-Sel-Cyclic-v7:1} and Corollary
\ref{corollary-Sel-Cyclic-v12:1}. \medskip

Let $\mathscr{F}_2(X)=\{S\in \mathscr{F}(X): |S|\leq 2\}$. A selection
$f:\mathscr{F}_2(X)\to X$ is commonly called a \emph{weak selection}
for $X$, see the next section for a brief review of such selections.
Theorem \ref{theorem-sa_points-v10:1} is not valid in the setting of
continuous weak selections. In this case, the property is
characterising another class of connected spaces which constitutes the
second main result of this paper. In order to state it, we briefly
recall some terminology. A ternary relation $\mathbf{C}\subset X^3$ on
a set $X$ is called a \emph{cyclic} (or \emph{circular}) \emph{order}
on $X$, see Huntington \cite{zbMATH02612155,zbMATH02595096} and
\v{C}ech \cite{cech:66}, if the following conditions are satisfied:
\begin{align*}
  \left.
  \begin{gathered}
    a\neq b\neq c\neq a\\
    (a,b,c)\notin \mathbf{C}
  \end{gathered}\right\} & \iff (c,b,a)\in \mathbf{C}\\
  (a,b,c)\in \mathbf{C} &\implies (b,c,a)\in \mathbf{C} \\
  \left.
  \begin{gathered}
    (a,b,c)\in \mathbf{C}\\
    (a,c,d)\in \mathbf{C}
  \end{gathered}
  \right\} &\implies (a,b,d)\in \mathbf{C}.
\end{align*}

Several basic properties of a cyclic order can be found in \S5 of
Chapter I in \cite{cech:66}. For instance, whenever $a\in X$, a cyclic
order $\mathbf{C}$ on $X$ defines a (strict) linear order
$<_{\mathbf{C},a}$ on $X\setminus\{a\}$ by $x<_{\mathbf{C},a} y$ if
$(a,x,y)\in \mathbf{C}$. The converse is also true, and each linear
order $<$ on $X\setminus\{a\}$ defines a unique cyclic order
$\mathbf{C}$ on $X$ with $<=<_{\mathbf{C},a}$. Furthermore, each
linear order $<$ on $X$ can be extended to a cyclic order
$\mathbf{C}_<$ on $X$, see \cite[Proposition 1.6]{MR0339099}, where
$\mathbf{C}_<$ is defined by
\[
  (a,b,c)\in \mathbf{C}_< \iff
  \begin{cases}
    a\neq b\neq c\neq a\quad \text{and}\\
    a<b<c\ \ \text{or}\ \ b<c<a\ \ \text{or}\ \ c<a<b. 
  \end{cases}
\]

For a cyclic order  $\mathbf{C}$ on $X$ and $a,b\in X$, the set 
\[
  (a,b)_{<_\mathbf{C}}=\{x\in X:
  (a,x,b)\in \mathbf{C}\}\subset X\setminus\{a,b\}
\]
is called an \emph{interval} from $a$ to $b$. Evidently,
$(a,b)_{<_\mathbf{C}}=(a,a)_{<_\mathbf{C}}=\emptyset$ provided
$a=b$. Otherwise, if $a\neq b$, it was shown in \cite{cech:66} that
the linear orders $<_{\mathbf{C},a}$ and $<_{\mathbf{C},b}$ coincide
on $(a,b)_{<_\mathbf{C}}$. A space $X$ is called \emph{weakly
  cyclically orderable} \cite{MR0339099} (\emph{cyclically orderable},
in Kok's terminology \cite{MR0339099}) if there exists a cyclic order
$\mathbf{C}$ on $X$ such that all intervals $(a,b)_{<_\mathbf{C}}$,
$a, b\in X$, are open in $X$. If, moreover, these intervals form a
base for the topology of $X$, then $X$ is called \emph{cyclically
  orderable} \cite{MR0339099} (\emph{strictly cyclically orderable},
in Kok's terminology \cite{MR0339099}). Each weakly orderable space is
weakly cyclically orderable \cite[Proposition
1.6]{MR0339099}. However, an orderable space is not necessarily
cyclically orderable. As pointed out in \cite{MR0339099}, such a space
is the half-open interval $[0,1)$. In the other direction, the plane
circle is an example of a connected cyclically orderable space which
is not weakly orderable. \medskip

If $X$ is weakly cyclically orderable and $p\in X$, then
$X\setminus\{p\}$ is weakly orderable; in fact, weakly orderable with
respect to the linear order $<_{\mathbf{C},p}$ \cite[Proposition
1.7]{MR0339099}. In particular, $X\setminus\{p\}$ has a continuous
weak selection. Evidently, this is a special case of the selection
property in Theorem \ref{theorem-sa_points-v10:1}. We are now ready to
state our second main result. Namely, in this paper, we will also
prove the following theorem which is complementary to Theorem
\ref{theorem-sa_points-v10:1}.

\begin{theorem}
  \label{theorem-Sel-Cyclic-v8:1}
  A connected space $X$ is weakly cyclically orderable if and only if
  $X\setminus\{p\}$ has a continuous weak selection, for every $p\in
  X$. 
\end{theorem}

The paper is organised as follows. The next section contains a brief
review of some results about continuous hyperspace selections in the
setting of connected spaces. Section \ref{sec:order-almost-weak}
contains a special case of Theorem \ref{theorem-sa_points-v10:1}, see
Theorem \ref{theorem-Sel-Cyclic-v2:1}, which is based on another
weaker interpretation of weak orderability. The proof of Theorem
\ref{theorem-sa_points-v10:1} is finalised in Section
\ref{sec:select-avoid-points}, while that of Theorem
\ref{theorem-Sel-Cyclic-v8:1} --- in Section
\ref{sec:weak-select-avoid}.  

\section{Connected Weakly Orderable Spaces} 
\label{sec:conn-weakly-order}

As mentioned in the Introduction, the weakly orderable spaces were
introduced by Eilenberg \cite{eilenberg:41}. In the same paper, he
obtained the following interesting result in the setting of connected
spaces.

\begin{theorem}[\cite{eilenberg:41}]
  \label{theorem-Sel-Cyclic-v10:1}
  Each connected weakly orderable space has precisely two compatible
  orders which are inverse to each other.
\end{theorem}

Ernest Michael was the first to relate linear orders to weak
selections. For a set $X$ and a weak selection
$\sigma:\mathscr{F}_2(X)\to X$, he defined a natural order-like
relation $\leq_\sigma$ on $X$ by $x\leq_\sigma y$ if
$\sigma(\{x,y\})=x$ \cite[Definition 7.1]{michael:51}. The relation
$\leq_\sigma$ is very similar to a linear order on $X$ being both
\emph{total} and \emph{antisymmetric}, but may fail to be
\emph{transitive}. For convenience, we write $x<_\sigma y$ provided
$x\leq_\sigma y$ and $x\neq y$. For a space $X$, the strict relation
$x<_\sigma y$ plays an important role in describing continuity of
$\sigma$. Namely, $\sigma$ is continuous iff for every $x,y\in X$ with
$x<_\sigma y$, there are open sets $U,V\subset X$ such that $x\in U$,
$y\in V$ and $s<_\sigma t$ for every $s\in U$ and $t\in V$
\cite[Theorem 3.1]{gutev-nogura:01a}.  Continuity of a weak selection
$\sigma$ implies that all $\leq_\sigma$-open intervals
$(\leftarrow, x)_{\leq_\sigma}=\{y\in X: y<_\sigma x\}$ and
$(x,\to)_{\leq_\sigma}= \{y\in X: x<_\sigma y\}$, $x\in X$, are open
in $X$ \cite{michael:51}, but the converse is not necessarily true
\cite[Example 3.6]{gutev-nogura:01a} (see also \cite[Corollary 4.2 and
Example 4.3]{gutev-nogura:09a}). For an extended review of (weak)
hyperspace selections, the interested reader is refereed to
\cite{gutev-2013springer}. Finally, let us explicitly remark that if
$f\in\sel[\mathscr{D}]$ for some $\mathscr{D}\subset \mathscr{F}(X)$
with $\mathscr{F}_2(X)\subset \mathscr{D}$, then
$f\uhr\mathscr{F}_2(X)$ is a continuous weak selection for $X$. In
this case, the order-like relation generated by
$f\uhr\mathscr{F}_2(X)$ will be simply denoted by $\leq_f$.\medskip

In the setting of connected spaces, using Theorem
\ref{theorem-Sel-Cyclic-v10:1}, Michael gave a complete description of
continuity of hyperspace selections. In the one direction, he showed
the following properties of these selections, see \cite[Lemmas 7.2 and
7.3]{michael:51}.

\begin{theorem}[\cite{eilenberg:41,michael:51}]
  \label{theorem-Sel-Cyclic-v3:1}
  Let $X$ be a connected space, $\mathscr{F}_2(X)\subset
  \mathscr{D}\subset \mathscr{F}(X)$ and $f\in \sel[\mathscr{D}]$. 
  Then 
  \begin{enumerate}
  \item $\leq_f$ is a linear order and $X$ is weakly orderable with
    respect to $\leq_f$.
  \item $f(S)=\min_{\leq_f} S$ for every $S\in \mathscr{D}$.
  \end{enumerate}
  Moreover, if $g\in\sel[\mathscr{D}]$ with $g\neq f$, then
  $g(S)=\max_{\leq_f}S$ for every $S\in \mathscr{D}$.  
\end{theorem}

Next, he established the following counterpart of Theorem
\ref{theorem-Sel-Cyclic-v3:1} in \cite[Lemma 7.5.1]{michael:51}, for
convenience we state it for the special case of
$\mathscr{D}=\mathscr{F}(X)$.

\begin{theorem}[\cite{michael:51}]
  \label{theorem-Sel-Cyclic-v4:1}
  Let $X$ be a connected space which is weakly orderable with respect
  to a linear order $\leq$ such that $\min_{\leq} S$ does exist, for
  each $S\in \mathscr{F}(X)$. Then $\mathscr{F}(X)$ has a continuous
  selection. In fact, $f(S)=\min_{\leq} S$, $S\in \mathscr{F}(X)$, is
  a continuous selection for $\mathscr{F}(X)$. 
\end{theorem}

Let $X$ be a connected space and $f\in \sel[\mathscr{F}(X)]$. Then by
Theorem \ref{theorem-Sel-Cyclic-v3:1}, $X$ is weakly orderable with
respect to $\leq_f$ and $p=f(X)$ is the $\leq_f$-minimal element of
$X$. However, $(p,\to)_{\leq_f}$ is a connected subset of $X$ being an
interval \cite[Theorem 1.3]{MR0339099}. Accordingly, $p$ has the
property that $X\setminus\{p\}$ is connected. Such a point $p\in X$ in
a connected space $X$ is called \emph{noncut}. Otherwise, if
$X\setminus \{p\}$ is not connected, the point $p$ is called
\emph{cut}.  Let $\noncut(X)$ be the set of all noncut points of $X$,
and $\cut(X)$ --- that of all cut points of $X$. It is well known that
each connected weakly orderable space has at most two noncut
points. The following further property is an immediate consequence of
the above considerations, see \cite[Theorem
1.1]{garcia-ferreira-gutev-nogura-sanchis-tomita:99}.

\begin{corollary}
  \label{corollary-sa_points-v6:2}
  Let $X$ be a connected space, $f\in \sel[\mathscr{F}(X)]$ and
  $p=f(X)$. Then  $p\in \noncut(X)$ and $f(S)=p$ for every
  $S\in\mathscr{F}(X)$ with $p\in S$. Moreover, if $q\in X$ with
  $q\neq p$, then $q\in\noncut(X)$ if and only if $f(S)\neq q$ for
  every $S\in \mathscr{F}(X)$ with $S\neq \{q\}$. 
\end{corollary}

\begin{proof}
  As remarked above, $p\in \noncut(X)$. Moreover, since $p$ is the
  $\leq_f$-minimal element of $X$, it follows from Theorem
  \ref{theorem-Sel-Cyclic-v3:1} that $f(S)=p$ for every
  $S\in \mathscr{F}(X)$ with $p\in S$.  Suppose that $q\in X$ with
  $q\neq p$. Then $X\setminus\{q\}$ is connected precisely when $q$ is
  the $\leq_f$-maximal element of $X$ \cite[Corollary 2.7]{gutev:07a},
  see also \cite[Theorem 1.3]{MR0339099}. Therefore, by Theorem
  \ref{theorem-Sel-Cyclic-v3:1}, $q\in \noncut(X)$ precisely when
  $f(S)\neq q$ for every $S\in \mathscr{F}(X)$ with $S\neq\{q\}$.
\end{proof}

There are connected weakly orderable spaces $X$ such that
$|\noncut(X)|=2$, but $\mathscr{F}(X)$ has precisely one continuous
selection for $\mathscr{F}(X)$. For instance, such a space is the
topologist's sine curve
$X = \{(0, 0)\} \cup\big\{(t, \sin 1/t) : 0 < t \leq 1\big\}$, see
\cite[Example 8 and Lemma 15]{nogura-shakhmatov:97a}. In this regards,
the following natural result was obtained by Nogura and Shakhmatov
\cite{nogura-shakhmatov:97a}.

\begin{theorem}[\cite{nogura-shakhmatov:97a}]
  \label{theorem-Sel-Cyclic-v4:2}
  A connected space $X$ is compact and orderable if and only if it has
  exactly two continuous selections for $\mathscr{F}(X)$.
\end{theorem}

\section{Selections Avoiding Noncut Points}
\label{sec:order-almost-weak}

In this section, $X$ is used to denote a connected space. If $p\in X$
is a cut point of $X$, then there are disjoint sets $U,V\subset X$
such that $X\setminus\{p\}=U\cup V$ and
$\{p\}=\overline{U}\cap \overline{V}$. Following \cite{MR3705772},
such a pair $(U,V)$ of sets will be called a \emph{$p$-cut} of $X$. A
point $p\in X$ is said to \emph{separate} $x,y\in X$ if $x\in U$ and
$y\in V$ for some $p$-cut $(U,V)$ of $X$. If $p$ separates $x$ and
$y$, then $p$ is a cut point of $X$, and neither $x$ nor $y$ separates
the other two points (see \cite[Lemma 2.1]{MR0339099}). In these
terms, $X$ is called \emph{almost weakly orderable} \cite[Definition
3.2]{MR3705772} if it has finitely many noncut points and among every
three points of $X$ with two of them being cut, there is one which
separates the other two.\medskip

A subset $E\subset X$ is called an \emph{endset of $X$} if
$X\setminus E$ is connected. Evidently, $p\in\noncut(X)$ precisely
when the singleton $\{p\}$ is an endset of $X$. Thus, noncut points
are often called \emph{endpoints}. However, a set of endpoints is not
necessarily an endset. In contrast, the endpoints of an almost weakly
orderable space form an endset, see \cite[Corollary
3.4]{MR3705772}. Based on this, we have the following alternative
interpretation of a special class of almost weakly orderable spaces.

\begin{proposition}
  \label{proposition-Sel-Cyclic-vgg:2}
  Let $X$ be a connected space such that $\noncut(X)$ is a nonempty
  finite set. Then $X$ is almost weakly orderable if and only if
  $\noncut(X)$ is an endset of $X$ such that $\{p\}\cup \cut(X)$ is
  weakly orderable for every $p\in \noncut(X)$. 
\end{proposition}

\begin{proof}
  Follows from the definition and the fact that a connected space $Z$
  is weakly orderable if and only if among every three points of $Z$
  there is one which separates the other two, see \cite[Theorem
  4.1]{MR0339099} (in a footnote of \cite{MR0235524}, the result was
  credited to D. Zaremba-Szczepkowicz).
\end{proof}

If $X$ is almost weakly
orderable, then there exists a partial order $\leq$ on $X$ such that
two points of $X$ are $\leq$-comparable precisely when they can be
separated, moreover this order is compatible with the topology of $X$
in the sense that all $\leq$-open intervals are open in $X$, see
\cite[Corollary 3.7]{MR3705772}. Such a partial order on $X$ is called
a \emph{separation partial order}, and any two separation partial
orderings on $X$ are either identical or inverse to each other
\cite[Proposition 3.8]{MR3705772}. Let us explicitly remark that the
idea of a separation order induced by cut points goes back to Whyburn
\cite{MR1501448}; the interested reader is also referred to
\cite{MR0125557, MR0264581}, and the more recent monograph
\cite{MR1192552}.\medskip

Let $X$ be almost weakly orderable.  It follows from Proposition
\ref{proposition-Sel-Cyclic-vgg:2}, see also \cite[Proposition
3.9]{MR3705772}, that for $p\in \noncut(X)$ and a separation partial
order $\leq$ on $X$, we have either $p< x$ for every $x\in\cut(X)$, or
$x<p$ for every $x\in\cut(X)$; in other words, $p<\cut(X)$ or
$\cut(X)< p$. In particular, noncut points $p,q\in X$ are not
$\leq$-comparable precisely when $\{p,q\}<\cut(X)$ or
$\cut(X)<\{p,q\}$. Accordingly, we have the following immediate
consequence.

\begin{corollary}
  \label{corollary-Sel-Cyclic-v11:1}
  If an almost weakly orderable space has more than two noncut points,
  then it has two noncut points which cannot be separated.
\end{corollary}

Using these observations, we will prove the following theorem. In the
proof of this theorem, for a set $Z$, a linear order $\leq$ on $Z$ and
$a,b\in Z$ with $a< b$, we let
\begin{equation}
  \label{eq:Sel-Cyclic-vgg:1}
  (a,b)_{\leq}=\{z\in Z: a<
  z< b\}.
\end{equation}

\begin{theorem}
  \label{theorem-Sel-Cyclic-v2:1}
  Let $X$ be a connected space which has at least one noncut point. If
  $\mathscr{F}(X\setminus\{p\})$ has a continuous selection for each
  $p\in\noncut(X)$, then $X$ is compact and orderable.
\end{theorem}

\begin{proof}
  Take a point $p\in \noncut(X)$ and a continuous selection $f$ for
  $\mathscr{F}(X\setminus\{p\})$. Since $X\setminus\{p\}$ is
  connected, by Theorem \ref{theorem-Sel-Cyclic-v3:1} and Corollary
  \ref{corollary-sa_points-v6:2}, $X\setminus\{p\}$ is weakly
  orderable with respect to $\leq_f$ and has a point
  $q\in X\setminus\{p\}$ with $q\leq_f x$ for every
  $x\in X\setminus\{p\}$, i.e.\ a noncut of $X\setminus\{p\}$. Then
  $q$ is also a noncut point of $X$ and, accordingly, $X$ has at least
  two noncut points. We will show that these are the only noncut
  points of $X$. Contrary to this, assume that $X$ has a noncut point
  $r\in X\setminus\{p,q\}$. Then $q<_f r$ and there exists a point
  $y\in X\setminus\{p\}$ with $q<_f y<_f r$. Moreover, both sets
  $(q,y)_{\leq_f}$ and $(y,r)_{\leq_f}$ are connected being intervals
  in the connected weakly orderable space $X\setminus\{p\}$
  \cite[Theorem 1.3]{MR0339099}. Hence, so are the sets
  $L=\overline{(q,y)_{\leq_f}}$ and $R=\overline{(y,r)_{\leq_f}}$. It
  is also evident that $q,y\in \noncut(L)$ and $y,r\in \noncut(R)$. We
  will show that this impossible. To this end, let us observe that
  $p\in \overline{(\gets, r)_{\leq_f}}$ because
  $p\notin \overline{(\gets, r)_{\leq_f}}$ will imply that
  $\overline{(\gets, r)_{\leq_f}}\setminus \{r\}=(\gets, r)_{\leq_f}$
  is a nonempty clopen proper subset of the connected space
  $X\setminus\{r\}$. Thus, we have that
  \[
    p\in \overline{(\gets,r)_{\leq_f}}=
    \overline{(q,y)_{\leq_f}\cup (y,r)_{\leq_f}}
    =\overline{(q,y)_{\leq_f}}\cup \overline{(y,r)_{\leq_f}} =
    L\cup R.
  \]
  In particular, $p\in \noncut(L)$ or $p\in \noncut(R)$. Hence, one of
  these sets has at least three noncut points. However, both
  $\mathscr{F}(L)$ and $\mathscr{F}(R)$ have continuous selections
  because so do $\mathscr{F}(X\setminus\{r\})$ and
  $\mathscr{F}(X\setminus \{q\})$, while
  $L\in \mathscr{F}(X\setminus\{r\})$ and
  $R\in \mathscr{F}(X\setminus\{q\})$. According to Theorem
  \ref{theorem-Sel-Cyclic-v3:1} and Corollary
  \ref{corollary-sa_points-v6:2}, this is impossible because each
  weakly orderable space has at most two noncut points. Thus, $p$ and
  $q$ are the only noncut points of $X$ and, in fact,
  $X\setminus\{p,q\}$ is connected. So, $\noncut(X)=\{p,q\}$ is an
  endset of $X$. Moreover, by Theorem \ref{theorem-Sel-Cyclic-v3:1},
  both $X\setminus\{p\}$ and $X\setminus\{q\}$ are weakly
  orderable. Therefore, by Proposition
  \ref{proposition-Sel-Cyclic-vgg:2}, $X$ is almost weakly
  orderable.\smallskip

  We are also ready to show that $X$ is compact and orderable. To see
  this, using one of these noncut points, for instance $p$, take a
  continuous selection $f$ for $\mathscr{F}(X\setminus\{p\})$. Then by
  Theorem \ref{theorem-Sel-Cyclic-v3:1}, $X\setminus\{p\}$ is weakly
  orderable with respect to $\leq_f$ and since
  $\cut(X)\subset X\setminus\{p\}$, it follows from \cite[Corollary
  3.7]{MR3705772} that $X$ has a separation partial order $\leq$ which
  extends the linear order $\leq_f$. This implies that $\leq$ is a
  linear order on $X$, i.e.\ that $p$ and $q$ are
  $\leq$-comparable. Indeed, by Theorem \ref{theorem-Sel-Cyclic-v3:1}
  and Corollary \ref{corollary-sa_points-v6:2}, $q\leq x$ for every
  $x\in X\setminus\{p\}$, i.e.\ $q< \cut(X)$. If $p$ and $q$ are not
  $\leq$-comparable, it follows from Corollary
  \ref{corollary-Sel-Cyclic-v11:1} that $\{p,q\}<\cut(X)$. In this
  case, take disjoint open sets $U,V\subset X$ with $p\in U$ and
  $q\in V$, and set $F=X\setminus (V\cup\{p\})$. Then
  $U\setminus\{p\}\subset F\subset \cut(X)$ which implies that for
  every $y\in \cut(X)$ there exists $x\in U\setminus\{p\}\subset F$
  with $x<y$, because $(\gets,y)_\leq\cap U$ is also a neighbourhood
  of $p$. However, by Theorem \ref{theorem-Sel-Cyclic-v3:1},
  $f(F)=\min_{\leq_f} F\in F\subset \cut(X)$, which is clearly
  impossible.  Thus, $X$ is weakly orderable with respect to $\leq$
  and, in particular, $q<\cut(X)<p$. This implies that $f$ can be
  extended to a continuous selection $h$ for $\mathscr{F}(X)$. Indeed,
  now each $S\in \mathscr{F}(X)$ has a $\leq$-minimal element because
  $p$ is the $\leq$-maximal element of $X$, so we may define
  $h(S)=f(S\setminus\{p\})=\min_{\leq_f}S= \min_{\leq} S$ for every
  $S\in \mathscr{F}(X)$ with $S\neq\{p\}$. According to Theorem
  \ref{theorem-Sel-Cyclic-v4:1}, $h$ is a continuous selection for
  $\mathscr{F}(X)$ with $h(X)=q$.  Interchanging $p$ and $q$, the same
  argument shows that $\mathscr{F}(X)$ also has a continuous selection
  which assigns to $X$ the point $p$. Therefore, by Theorem
  \ref{theorem-Sel-Cyclic-v4:2}, $X$ is compact and orderable.
\end{proof}

\section{Selections Avoiding Points}
\label{sec:select-avoid-points}

Here, we finalise the proof of Theorem \ref{theorem-sa_points-v10:1},
which is based on the following special type of hyperspace
selections. For a space $X$ and $p\in X$, a selection $f$ for
$\mathscr{F}(X)$ is called \emph{$p$-minimal}
\cite{garcia-ferreira-gutev-nogura-sanchis-tomita:99} if $f(S)\neq p$
for every $S\in \mathscr{F}(X)$ with $S\neq\{p\}$. The prototype of
the following property can be found in \cite[Theorem
3.1]{gutev-nogura:00d}.

\begin{proposition}
  \label{proposition-sa_points-v10:1}
  Let $X$ be a space which has a $p$-minimal selection
  $f\in \sel[\mathscr{F}(X)]$ for some point $p\in X$.  Then
  $f(S\cup\{p\})\in S$, whenever $S\subset X$ is a nonempty subset
  with ${S\cup\{p\}\in \mathscr{F}(X)}$. In particular,
  $\sel[\mathscr{F}(X\setminus\{p\})]\neq \emptyset$.
\end{proposition}

\begin{proof}
  Let $S\subset X$ be a nonempty set with
  $S\cup\{p\}\in \mathscr{F}(X)$. Then $f(S\cup\{p\})=p$ precisely
  when $S\cup\{p\}=\{p\}$, i.e.\ $S=\{p\}$, because $f$ is
  $p$-minimal. This is equivalent to $f(S\cup\{p\})\in S$. For the
  second part of this proposition, let us observe that
  $\overline{S}\subset S\cup\{p\}$, whenever
  $S\in \mathscr{F}(X\setminus\{p\})$. Hence, we may define a map
  $\varphi:\mathscr{F}(X\setminus\{p\})\to \mathscr{F}(X)$ by
  \[
    \varphi(S)=S\cup\{p\},\quad \text{for every $S\in
      \mathscr{F}(X\setminus\{p\})$.}
  \]
  Take ${S\in \mathscr{F}(X\setminus\{p\})}$ and a finite family
  $\mathscr{V}$ of open subsets of $X$ with
  $\varphi(S)\in \langle\mathscr{V}\rangle$. Next, take another finite
  family $\mathscr{U}$ of nonempty open subsets of $X$ such that
  $\mathscr{U}$ refines $\mathscr{V}$ and
  $\bigcup\mathscr{U}=\bigcup\mathscr{V}\setminus \{p\}$. Then
  $S\in \langle\mathscr{U}\rangle$ and
  $\varphi(\langle\mathscr{U}\rangle)\subset
  \langle\mathscr{V}\rangle$, so $\varphi$ is continuous with respect
  to the Vietoris topology on these hyperspaces. We may now define a
  continuous selection $g$ for $\mathscr{F}(X\setminus \{p\})$ by
  $g=f\circ \varphi$.
\end{proof}

\begin{proof}[Proof of Theorem \ref{theorem-sa_points-v10:1}]
  Assume that $X$ is a compact connected space which is
  orderable. Next, take a compatible linear order $\leq$ on $X$ and a
  point $p\in X$.  If $p$ is a noncut point of $X$, then by Corollary
  \ref{corollary-sa_points-v6:2} and Theorem
  \ref{theorem-Sel-Cyclic-v4:2}, $\mathscr{F}(X)$ has a continuous
  $p$-minimal selection. Hence, by Proposition
  \ref{proposition-sa_points-v10:1},
  $\sel[\mathscr{F}(X\setminus\{p\})]\neq \emptyset$. If
  $p\in \cut(X)$, then $p$ is a noncut point for both intervals
  \[
    Y=\{x\in X: x\leq p\}\quad\text{and}\quad Z=\{x\in X:p\leq x\}.
  \]
  Moreover, these intervals are infinite, compact and
  orderable. Hence, for the same reason as before,
  $\sel[\mathscr{F}(Y\setminus\{p\})]\neq \emptyset\neq
  \sel[\mathscr{F}(Z\setminus \{p\})]$.  Since $Y\setminus\{p\}$ and
  $Z\setminus\{p\}$ form a clopen partition of $X\setminus\{p\}$, we
  also have that
  $\sel[\mathscr{F}(X\setminus\{p\})]\neq \emptyset$.\medskip

  To see the converse, by Theorem \ref{theorem-Sel-Cyclic-v2:1}, it
  suffices to show that $X$ has a noncut point provided
  $\sel[\mathscr{F}(X\setminus\{p\})]\neq \emptyset$ for each
  $p\in X$. To this end, take cut points $y,z\in \cut(X)$. Next, let
  $(A,B)$ be a $y$-cut of $X$, and $(U,V)$ be a $z$-cut of $X$. Then
  $z$ doesn't belong to one of the sets $A$ or $B$, say $z\notin
  A$. Since $S=A\cup\{y\}$ is a connected subset of $X$, it is a
  subset of $U$ or $V$, for instance $S\subset U$. Moreover, it is
  closed in $X$. This implies that $S$ has a noncut point $p\in S$
  with $p\neq y$.  Indeed, by Corollary
  \ref{corollary-sa_points-v6:2}, $S$ has a noncut point because
  $\mathscr{F}(X\setminus\{z\})$ has a continuous selection, and hence
  so does $\mathscr{F}(S)$. If this point is $y$, then
  $A=S\setminus\{y\}$ is connected and $\mathscr{F}(A)$ has a
  continuous selection because so does
  $\mathscr{F}(X\setminus\{y\})$. Therefore, $A$ has a noncut point
  $p\in A$. Since $S$ is weakly orderable, $p$ is also a noncut point
  of $S$. Thus, $p\neq y$ and $S\setminus\{p\}$ is connected. However,
  $B\cup\{y\}$ is also connected and $y\in S\setminus\{p\}$, so
  $X\setminus\{p\}=(S\setminus\{p\})\cup B\cup \{y\}$ is connected as
  well, i.e.\ $p\in\noncut(X)$. We may now apply Theorem
  \ref{theorem-Sel-Cyclic-v2:1} to complete the proof.
\end{proof}

We conclude this section with some observations regarding the proper
place of Theorem \ref{theorem-sa_points-v10:1}. To this end, we will
first establish the following partial inverse of Proposition
\ref{proposition-sa_points-v10:1}.

\begin{proposition}
  \label{proposition-Sel-Cyclic-v7:1}
  Let $X$ be a space and $p\in X$ be a point which is a countable
  intersection of clopen sets. Then the following are
  equivalent\textup{:}
  \begin{enumerate}
  \item\label{item:Sel-Cyclic-v7:1}
    $\sel[\mathscr{F}(X)]\neq \emptyset$,
  \item\label{item:Sel-Cyclic-v7:2} $\mathscr{F}(X)$ has a continuous
    $p$-minimal selection,
  \item\label{item:Sel-Cyclic-v7:3}
    $\sel[\mathscr{F}(X\setminus\{p\})]\neq \emptyset$.
  \end{enumerate}
\end{proposition}

\begin{proof}
  The implication
  \ref{item:Sel-Cyclic-v7:1}$\implies$\ref{item:Sel-Cyclic-v7:2}
  follows from \cite[Proposition 3.6]{gutev-nogura:08b}, while that of
  \ref{item:Sel-Cyclic-v7:2}$\implies$\ref{item:Sel-Cyclic-v7:3} is a
  consequence of Proposition \ref{proposition-sa_points-v10:1}. The
  implication
  \ref{item:Sel-Cyclic-v7:3}$\implies$\ref{item:Sel-Cyclic-v7:1} is
  implicitly contained in the proof of \cite[Proposition
  3.6]{gutev-nogura:08b}. Briefly, take a selection
  $f\in\sel[\mathscr{F}(X\setminus\{p\})]$ and a decreasing clopen
  family $\{V_{n}:n<\omega\}$ with $V_{0}=X$ and
  $\{p\}=\bigcap_{n<\omega}V_{n}$. Set $S_{n}=V_{n}\setminus V_{n+1}$,
  $n<\omega$, and for every $F\in \mathscr{F}(X)$ with $F\neq \{p\}$,
  let $n(F)=\min\{n<\omega:F\cap S_{n}\neq\emptyset\}$.  Finally,
  define a $p$-minimal selection $g:\mathscr{F}(X)\to X$ by
  $g(F)=f\big(F\cap S_{n(F)}\big)$, whenever $F\in \mathscr{F}(X)$
  with $F\neq \{p\}$; this is correct because
  $F\cap S_{n(F)}\in \mathscr{F}(X\setminus\{p\})$.  Moreover, $g$ is
  continuous at $\{p\}$ because each selection is continuous on the
  singletons. Take $F\in \mathscr{F}(X)$ with $F\neq \{p\}$, and a
  neighbourhood $U$ of $g(F)$. Since
  $f\uhr \mathscr{F}\big(S_{n(F)}\big)$ is continuous, there exists a
  finite family $\mathscr{W}_{0}$ of nonempty open subsets of
  $S_{n(F)}$ such that
  $F\cap S_{n(F)}\in \langle \mathscr{W}_{0}\rangle \subset
  f^{-1}(U)$. Set
  $W_0=V_{n(F)+1}\cup \big(\bigcup \mathscr{W}_{0}\big)$ and
  $\mathscr{W}=\{W_0\}\cup \mathscr{W}_{0}$. Then
  $\langle\mathscr{W}\rangle$ is a $\tau_{V}$-neighbourhood of $F$
  with $g(\langle \mathscr{W}\rangle)\subset U$. Indeed,
  $T\in \langle \mathscr{W}\rangle$ implies
  $T\cap S_{n(F)}\neq \emptyset$ and $T\subset V_{n(F)}$, so
  $n(T)=n(F)$. On the other hand,
  $T\cap S_{n(F)}\in \langle\mathscr{W}_0\rangle$ and therefore
  $g(T)=f\big(T\cap S_{n(T)}\big)\in U$. The proof is complete. 
\end{proof}

Based on this, we have the following theorem which is complementary to
Theorem \ref{theorem-sa_points-v10:1}. In this theorem, a space $X$ is
\emph{totally disconnected} if each singleton of $X$ is an
intersection of clopen subsets of $X$. Also, let us recall that a
space $X$ is of \emph{countable tightness} if for each $A \subset X$
and $x \in \overline{A}$, there exists a countable set $B \subset A$
with $x \in \overline{B}$.

\begin{theorem}
  \label{theorem-Sel-Cyclic-v7:1}
  Let $X$ be a totally disconnected space which is of countable
  tightness. Then $\sel[\mathscr{F}(X)]\neq \emptyset$ if and only if 
  $\sel[\mathscr{F}(X\setminus\{p\})]\neq \emptyset$, for every
  $p\in X$.
\end{theorem}

\begin{proof}
  Take points $p,q\in X$ and a clopen set $Y\subset X$ with $p\in Y$
  and $q\notin Y$. If $\mathscr{F}(X)$ has a continuous selection,
  then so does $\mathscr{F}(Y)$ because $Y\in
  \mathscr{F}(X)$. Similarly, $\mathscr{F}(Y)$ has a continuous
  selection provided so does $\mathscr{F}(X\setminus\{q\})$. The proof
  now consists of showing that if
  $\sel[\mathscr{F}(Y)]\neq \emptyset$, then
  $\sel[\mathscr{F}(X)]\neq \emptyset$ precisely when
  $\sel[\mathscr{F}(X\setminus\{p\})]\neq \emptyset$. To this end, by
  Proposition \ref{proposition-Sel-Cyclic-v7:1}, it suffices to show
  that $p$ is a countable intersection of clopen sets of $Y$. So, take
  a selection $f\in \sel[\mathscr{F}(Y)]$. Then the above property is
  reduced to show that $p$ is a countable intersection of relatively
  clopen sets in each one of the $\leq_f$-intervals
  \[
    (\gets,p]_{\leq_f}=\{y\in Y: y\leq_f p\}\quad \text{and}\quad
    [p,\to)_{\leq_f}=\{y\in Y:p\leq_f y\}.
  \]
  If $p$ is a non-isolated point of $(\gets,p]_{\leq_f}$, using that
  $X$ is of countable tightness, there is a countable set
  $A\subset (\gets,p)_{\leq_f}$ with $p\in\overline{A}$. Therefore,
  $\bigcap_{x\in A}(x,p]_{\leq_f}=\{p\}$, see \cite[Theorem
  4.1]{garcia-ferreira-gutev-nogura-sanchis-tomita:99} and
  \cite[Remark 3.5]{MR3478342}. According to \cite[Proposition
  5.6]{gutev-nogura:09a}, this implies that $p$ is a countable
  intersection of clopen subsets of $(\gets,p]_{\leq_f}$ because $X$
  is totally disconnected. Similarly, $p$ is also a countable
  intersection of clopen subsets of $[p,\to)_{\leq_f}$. The proof is
  complete.
\end{proof}

According to \cite[Corollary 5.3]{MR3712970}, a space $X$ is totally
disconnected provided $\mathscr{F}(Z)$ has a continuous selection for
every $Z\subset X$ with $|X\setminus Z| \leq 2$. Applying twice
Theorem \ref{theorem-Sel-Cyclic-v7:1}, this gives the following
immediate consequence.

\begin{corollary}
  \label{corollary-Sel-Cyclic-v12:1}
  For a space $X$ which is of countable tightness, the following
  conditions are equivalent\textup{:}
  \begin{enumerate}
  \item\label{item:Sel-Cyclic-v12:1} $X$ is totally disconnected and
    $\sel[\mathscr{F}(X)]\neq\emptyset$.
  \item\label{item:Sel-Cyclic-v12:2}
    $\sel[\mathscr{F}(X\setminus S)]\neq \emptyset$, for every
    $S\in \mathscr{F}_2(X)$.
  \end{enumerate}
\end{corollary}

Evidently, by a finite induction, \ref{item:Sel-Cyclic-v12:2} of
Corollary \ref{corollary-Sel-Cyclic-v12:1} can be extended to all
nonempty finite subsets $S\subset X$. Regarding this, let us remark
that if $X$ is a regular space such that $\mathscr{F}(Z)$ has a
continuous selection for every nonempty open $Z\subset X$, then $X$
has a clopen $\pi$-base \cite[Theorem 5.4]{MR3712970}. Here, a family
$\mathscr{P}$ of nonempty open subsets of $X$ is a \emph{$\pi$-base}
for $X$, or a \emph{pseudobase}, if each nonempty open subset of $X$
contains an element of $\mathscr{P}$.\medskip

The hypothesis in Theorem \ref{theorem-sa_points-v10:1} that $X$ is
connected is essential to conclude that $X$ is compact. Namely, each
completely metrizable space $X$ which has a covering dimension zero,
i.e.\ being \emph{strongly zero-dimensional}, has a continuous
selection for $\mathscr{F}(X)$
\cite{choban:70a,engelking-heath-michael:68}. Hence, by Theorem
\ref{theorem-Sel-Cyclic-v7:1}, if $X$ is a strongly zero-dimensional
completely metrizable space, then
$\sel[\mathscr{F}(X\setminus\{p\})]\neq \emptyset$ for every $p\in
X$. In fact, Theorem \ref{theorem-Sel-Cyclic-v7:1} is not so relevant
in this case. It is well known that complete metrizability is
inherited on $G_\delta$-sets. Moreover, such sets remain strongly
zero-dimensional provided so is $X$. Thus, in the setting of a
strongly zero-dimensional completely metrizable space $X$, the
hyperspace $\mathscr{F}(Z)$ has a continuous selection for every
nonempty $G_\delta$-subset $Z\subset X$. Based on this, the following
question was posed in \cite[Question 5]{MR3712970}, it is still
open.

\begin{question}[\cite{MR3712970}]
  \label{question-Sel-Cyclic-v12:1}
  Let $X$ be a (completely) metrizable space with the property that
  $\mathscr{F}(Z)$ has a continuous selection for every nonempty
  $G_\delta$-subset $Z\subset X$. Then, is it true that $X$ is
  strongly zero-dimensional?
\end{question}

Going back to the selection property that
``$\sel[\mathscr{F}(X\setminus\{p\})]\neq \emptyset$ for every point
${p\in X}$'', let us remark that it always implies the existence of a
continuous selection for the nonempty closed subsets of $X$.  Namely,
answering a question in a previous version of this paper, the
following observation was communicated to the author by Jorge Antonio
Cruz Chapital.

\begin{proposition}
  \label{proposition-Sel-Cyclic-vgg:1}
  Let $X$ be a space such that $\sel[\mathscr{F}(X\setminus\{p\})]\neq
  \emptyset$, for every $p\in X$. Then $\sel[\mathscr{F}(X)]\neq
  \emptyset$. 
\end{proposition}

\begin{proof}
  If $X$ is connected, this follows from Theorem
  \ref{theorem-sa_points-v10:1}. Otherwise, if $X$ is not connected,
  then it has a nonempty clopen proper subset $A\subset X$. Taking
  points $p\in A$ and $q\in B=X\setminus A$, it follows that
  $A\in \mathscr{F}(X\setminus\{q\})$ and
  $B\in \mathscr{F}(X\setminus\{p\})$, therefore
  $\sel[\mathscr{F}(A)]\neq \emptyset$ and
  $\sel[\mathscr{F}(B)]\neq \emptyset$. Since the sets $A$ and $B$
  form a clopen partition of $X$, we also have that
  $\sel[\mathscr{F}(X)]\neq \emptyset$.
\end{proof}

The condition in Proposition \ref{proposition-Sel-Cyclic-vgg:1} that
$\mathscr{F}(X\setminus\{p\})$ has a continuous selection for every
$p\in X$ is important. Indeed, one can easily construct examples of
compact connected metrizable spaces which have a continuous selection
for $\mathscr{F}(X\setminus\{p\})$ for some point $p\in X$, but
$\sel[\mathscr{F}(X)]=\emptyset$. For instance, take any simple triod
$X$, i.e.\ the union of 3 arcs having a common endpoint $p\in X$ and
being mutually disjoint except at that point.

\section{Weak Selections Avoiding Points}
\label{sec:weak-select-avoid}

In this section, we will prove Theorem \ref{theorem-Sel-Cyclic-v8:1},
which is based on known results and the following special case of this
theorem (compare with \cite[Theorem 3.11]{MR0339099}).

\begin{lemma}
  \label{lemma-Sel-Cyclic-v12:1}
  Let $X$ be a connected space such that
  $\sel[\mathscr{F}_2(X\setminus\{p\})]\neq \emptyset$, for every
  $p\in X$. If $X$ has a cut point, then it is weakly orderable. 
\end{lemma}

\begin{proof}
  Suppose that $X$ has a cut point $q\in X$, and take a $q$-cut
  $(U,V)$ of $X$. Then $C=U\cup\{q\}$ is a connected set. Hence, by
  Theorem \ref{theorem-Sel-Cyclic-v3:1}, it is weakly orderable
  because $C\subset X\setminus\{p\}$ for every (some) point $p\in V$,
  and $X\setminus\{p\}$ has a continuous weak selection.  We are going
  to show that $q$ is a noncut point of $C$. To this end, suppose that
  $q$ is a cut point of $C$, and take another cut point $r\in U$ of
  $C=U\cup\{q\}$, also an $r$-cut $(A,B)$ of $C$ with $q\in A$. Since
  $C$ is weakly orderable, $A$ is connected and $q$ is a cut point of
  $A$ as well. Moreover, $A$ is also weakly orderable. Let $\leq_A$ be
  a compatible linear order on $A$ and $a,b\in A$ be such that
  \begin{equation}
    \label{eq:Sel-Cyclic-v9:1}
    a<_A q<_A b.
  \end{equation}
  Next, for convenience, set $E=(a,b)_{\leq_A}\subset A$, see
  \eqref{eq:Sel-Cyclic-vgg:1}, which is a connected set being an
  interval, see \cite[Theorem 1.3]{MR0339099}. Finally, let
  \[
    D=V\cup\{q\}\subset X\setminus\{r\}\quad \text{and}\quad Y=A\cup
    V=A\cup D.
  \]
  Then $Y$ is a connected subset of $X\setminus\{r\}$ because
  $q\in A\cap D$ and $D$ is connected.  Hence, for the same reason as
  before, $Y$ is weakly orderable. Since both $A$ and $Y$ are
  connected, by Theorem \ref{theorem-Sel-Cyclic-v10:1}, $Y$ has a
  compatible linear order $\leq$ with ${\leq\uhr A=\leq_A}$. We now
  have that $E=(a,b)_\leq=\{y\in Y: a<y< b\}$ because
  $y\in Y\setminus E$ implies that $E\subset (\gets,y)_\leq$ or
  $E\subset (y,\to)_\leq$, see e.g.\ \cite[Proposition
  2.6]{gutev:07a}. Thus, $E$ is an open subset of $Y$ and
  $V\subset Y\setminus E$, therefore
  $D=\overline{V}\subset Y\setminus E$. However, this is impossible
  because $q\in E\cap D$. A contradiction! \smallskip

  Evidently, the same reasoning applies to show that $q$ is also a
  noncut point of $D$. Since $C$ and $D$ are weakly orderable, so is
  the space $X=C\cup D$.
\end{proof}

\begin{proof}[Proof of Theorem \ref{theorem-Sel-Cyclic-v8:1}]
  If $X$ is weakly cyclically orderable and $p\in X$, then
  $X\setminus\{p\}$ is weakly orderable \cite[Proposition
  1.7]{MR0339099}. Accordingly, $X\setminus \{p\}$ has a continuous
  weak selection. Conversely, suppose that $X\setminus \{p\}$ has a
  continuous weak selection, for each $p\in X$. To show that $X$ is
  weakly cyclically orderable, we distinguish the following two
  cases. If $\noncut(X)=X$, take any point $p\in X$ and a nonempty
  connected subset $Y\subset X\setminus\{p\}$. By hypothesis,
  $X\setminus\{p\}$ has a continuous weak selections, hence so does
  $Y$. Accordingly, by Theorem \ref{theorem-Sel-Cyclic-v3:1}, $Y$ is
  weakly orderable which implies that it has at most two noncut
  points, see e.g.\ \cite[Theorem 3.5]{MR0339099}. Therefore, by
  \cite[Theorem 3.18]{MR0339099}, $X$ is weakly cyclically
  orderable. If $X$ has a cut point, it follows from Lemma
  \ref{lemma-Sel-Cyclic-v12:1} that $X$ is weakly orderable, hence it
  is weakly cyclically orderable as well \cite[Proposition
  1.6]{MR0339099}. 
\end{proof}

We conclude this paper with the following consequence of Theorem
\ref{theorem-Sel-Cyclic-v8:1}, which provides a natural generalisation
of \cite[Proposition 1.7]{MR0339099}.

\begin{corollary}
  \label{corollary-Sel-Cyclic-v12:2}
  A connected space $X$ is weakly cyclically orderable if and only if
  $X\setminus\{p\}$ is weakly orderable, for every $p\in X$.
\end{corollary}

\begin{proof}
  If $X$ is weakly cyclically orderable, then by \cite[Proposition
  1.7]{MR0339099}, $X\setminus\{p\}$ is weakly orderable for each
  point $p\in X$. If $X\setminus\{p\}$ is weakly orderable for each
  point $p\in X$, then each $X\setminus\{p\}$, $p\in X$, has a
  continuous weak selection. Hence, by Theorem
  \ref{theorem-Sel-Cyclic-v8:1}, $X$ is weakly cyclically orderable. 
\end{proof}


\begin{thebibliography}{10}

\bibitem{zbMATH03379800}
A.~E. Brouwer, \emph{{A characterization of connected (weakly) orderable
  spaces.}}, {Math. Centrum, Amsterdam, Afd. zuivere Wisk. ZW 10/71, 7 p.
  (1971).}

\bibitem{zbMATH03355968}
\bysame, \emph{{On the topological characterization of the real line.}},
  {Math. Centrum, Amsterdam, Afd. zuivere Wisk. ZW 8/71, 6 p. (1971).}

\bibitem{MR3705772}
D.~Buhagiar and V.~Gutev, \emph{Selections and deleted symmetric products},
  Tsukuba J. Math. \textbf{41} (2017), no.~1, 1--20.

\bibitem{cech:66}
E.~{\v{C}}ech, \emph{Topological spaces}, Revised edition by Zden\v ek Frol\'\i
  k and Miroslav Kat\v etov. Scientific editor, Vlastimil Pt\'ak. Editor of the
  English translation, Charles O. Junge, Publishing House of the Czechoslovak
  Academy of Sciences, Prague, 1966.

\bibitem{choban:70a}
M.~Choban, \emph{Many-valued mappings and {B}orel sets. {I}}, Trans. Moscow
  Math. Soc. \textbf{22} (1970), 258--280.

\bibitem{dalen-wattel:73}
J.~van Dalen and E.~Wattel, \emph{A topological characterization of ordered
  spaces}, Gen. Top. Appl. \textbf{3} (1973), 347--354.
  
\bibitem{MR0235524}
R.~Duda, \emph{On ordered topological spaces}, Fund. Math. \textbf{63} (1968),
  295--309.
  
\bibitem{eilenberg:41}
S.~Eilenberg, \emph{Ordered topological spaces}, Amer. J. Math. \textbf{63}
  (1941), 39--45.

\bibitem{engelking-heath-michael:68}
R.~Engelking, R.~W. Heath, and E.~Michael, \emph{Topological well-ordering and
  continuous selections}, Invent. Math. \textbf{6} (1968), 150--158.

\bibitem{garcia-ferreira-gutev-nogura-sanchis-tomita:99}
S.~Garc{\'\i}a-Ferreira, V.~Gutev, T.~Nogura, M.~Sanchis, and A.~Tomita,
  \emph{Extreme selections for hyperspaces of topological spaces}, Topology
  Appl. \textbf{122} (2002), 157--181.

\bibitem{gutev:07a}
V.~Gutev, \emph{Weak orderability of second countable spaces}, Fund. Math.
  \textbf{196} (2007), no.~3, 275--287.

\bibitem{gutev-2013springer}
\bysame, \emph{Selections and hyperspaces}, Recent progress in general
  topology, {III} (K.~P. Hart, J.~van Mill, and P.~Simon, eds.), Atlantis
  Press, Springer, 2014, pp.~535--579.

\bibitem{MR3478342}
\bysame, \emph{Selections and approaching points in products}, Comment. Math.
  Univ. Carolin. \textbf{57} (2016), no.~1, 89--95.

\bibitem{MR3712970}
\bysame, \emph{Scattered spaces and selections}, Topology Appl. \textbf{231}
  (2017), 306--315.

\bibitem{gutev-nogura:01a}
V.~Gutev and T.~Nogura, \emph{Selections and order-like relations}, Appl. Gen.
  Topol. \textbf{2} (2001), 205--218.

\bibitem{gutev-nogura:00d}
\bysame, \emph{Fell continuous selections and topologically well-orderable
  spaces}, Mathematika \textbf{51} (2004), 163--169.

\bibitem{gutev-nogura:08b}
\bysame, \emph{Set-maximal selections}, Topology Appl. \textbf{157} (2010),
  53--61.

\bibitem{gutev-nogura:09a}
\bysame, \emph{Weak orderability of topological spaces}, Topology Appl.
  \textbf{157} (2010), 1249--1274.

\bibitem{MR0125557}
J.~G. Hocking and G.~S. Young, \emph{Topology}, Addison-Wesley Publishing Co.,
  Inc., Reading, Mass.-London, 1961.
  
\bibitem{zbMATH02612155}
E.~V. {Huntington}, \emph{A set of independent postulates for cyclic order},
  Proc Natl Acad Sci U S A \textbf{2} (1916), 630--631.

\bibitem{zbMATH02595096}
\bysame, \emph{Sets of completely independent postulates for cyclic order},
  Proc Natl Acad Sci U S A \textbf{10} (1924), 74--78.

\bibitem{MR0339099}
H.~Kok, \emph{Connected orderable spaces}, Mathematisch Centrum, Amsterdam,
  1973, Mathematical Centre Tracts, No. 49.

\bibitem{michael:51}
E.~Michael, \emph{Topologies on spaces of subsets}, Trans. Amer. Math. Soc.
  \textbf{71} (1951), 152--182.

\bibitem{mill-wattel:81}
J.~van Mill and E.~Wattel, \emph{Selections and orderability}, Proc. Amer. Math.
  Soc. \textbf{83} (1981), no.~3, 601--605.

\bibitem{MR1192552}
S.~B. Nadler, Jr., \emph{Continuum theory}, Monographs and Textbooks in Pure
  and Applied Mathematics, vol. 158, Marcel Dekker, Inc., New York, 1992, An
  introduction.
  
\bibitem{nogura-shakhmatov:97a}
T.~Nogura and D.~Shakhmatov, \emph{Characterizations of intervals via
  continuous selections}, Rendiconti del Circolo Matematico di Palermo, Serie
  II, \textbf{46} (1997), 317--328.

\bibitem{MR1501448}
G.~T. Whyburn, \emph{Concerning the cut points of continua}, Trans. Amer. Math.
  Soc. \textbf{30} (1928), no.~3, 597--609.

\bibitem{MR0264581}
S.~Willard, \emph{General topology}, Addison-Wesley Publishing Co., Reading,
  Mass.-London-Don Mills, Ont., 1970.
  
\end{thebibliography}
\end{document}